\documentclass[12pt,twoside,a4paper]{article}

\usepackage{color}
\usepackage{bbm}
\usepackage{graphicx}
\usepackage{amsfonts,amssymb,amsmath,amsthm,mathrsfs,enumerate,multicol,esint}

\newtheorem{theorem}{Theorem}[section]
\newtheorem{prop}[theorem]{Proposition}
\newtheorem{lemma}[theorem]{Lemma}
\newtheorem{definition}[theorem]{Definition}

\newtheorem{remark}[theorem]{Remark}

\allowdisplaybreaks

\begin{document}
\arraycolsep=1pt

\title{Weak Factorizations of the Hardy space $H^1(\mathbb{R}^n)$ in terms of Multilinear Riesz Transforms}
%\thanks{ BDW's research supported in part by National Science Foundation DMS \# 0955432}
%\thanks{{\it {\rm 2010} Mathematics Subject Classification:} Primary: 42B35, 42B25.}
%\thanks{{\it Key words:}
% Hardy space, BMO space, multilinear Riesz transform, weak factorization.}

\author{Ji Li and Brett D. Wick}

\date{}
\maketitle

\begin{abstract}
This paper provides a constructive proof of the weak factorizations of the classical Hardy space $H^1(\mathbb{R}^n)$ in terms of multilinear Riesz transforms. As a direct application, we obtain a new proof of the characterization of ${\rm BMO}(\mathbb{R}^n)$ (the dual of $H^1(\mathbb{R}^n)$) via commutators of the multilinear Riesz transforms.
\end{abstract}

\bigskip
\bigskip


\medskip


 %\tableofcontents

\bigskip

\section{Introduction and Statement of Main Results}
\setcounter{equation}{0}

The real-variable Hardy space theory on $n$-dimensional Euclidean space $\mathbb{R}^n$ ($n\geq1$) plays an important role in harmonic analysis and has been systematically developed.  An important result about the Hardy space is the weak factorization obtained by Coifman, Rochberg and Weiss \cite{CRW}.  This factorization proves that all $H^1(\mathbb{R}^n)$ can be written in terms of bilinear forms associated to the Riesz transforms, with the basic building blocks being:
$$
\Pi_j(f,g)=f R_j g+gR_jf,
$$
with $R_j$ the $j$th Riesz transform $\displaystyle R_jf(x)=\int_{\mathbb{R}^n} f(y) \frac{x_j-y_j}{\left\vert x-y\right\vert^{n+1}}\,dy$.  This result follows as a corollary of the characterization of the function space ${\rm BMO}(\mathbb{R}^n)$ in terms of the boundedness of the commutators $[b,R_j](f)=bR_j f-R_j(bf)$.

The main goals of this paper are to provide a constructive proof of the weak factorizations of the classical Hardy space $H^1(\mathbb{R}^n)$ in terms of multilinear Riesz transforms. As a direct corollary, we obtain a full characterization of ${\rm BMO}(\mathbb{R}^n)$ (the dual of $H^1(\mathbb{R}^n)$) via commutators of the multilinear Riesz transforms.  Our strategy and approach will be to modify the direct constructive proof of Uchiyama in \cite{U} for the weak factorization of the Hardy spaces.

We now recall the multilinear Calder\'on--Zygmund operators (see for example the statements in \cite{GT}).  Let $K(y_0,y_1,\ldots,y_m)$ be a locally integrable function defined away from the diagonal
$\{ y_0=y_1=\cdots=y_m \}$. $K$ is said to be an $m$-linear Calder\'on--Zygmund kernel if there exist positive constants $A$ and $\epsilon$ such that 
\begin{align}\label{size}
|K(y_0,y_1,\ldots,y_m)| \leq {A\over \big( \sum_{k,l=0} |y_k-y_l| \big)^{mn}}
\end{align}
and
\begin{align}\label{regularity}
|K(y_0,y_1,\ldots,y_j, \ldots,y_m) - K(y_0,y_1,\ldots,y_j', \ldots,y_m)|\leq {A |y_j-y'_j|^\epsilon\over \big( \sum_{k,l=0} |y_k-y_l| \big)^{mn+\epsilon}}
\end{align}
for all $0\leq j\leq m$ and $|y_j-y'_j|\leq {1\over2} \max_{0\leq k\leq m} |y_j-y_k|$.

Suppose $T$ is an $m$-linear operator defined on $L^{p_1}(\mathbb{R}^n)\times \cdots \times L^{p_m}(\mathbb{R}^n)$ associated with the $m$-linear Calder\'on--Zygmund kernel $K$, i.e.,
\begin{align}
T(f_1,\ldots,f_m)(x) := \int_{\mathbb{R}^{mn}} K(x,y_1,\ldots,y_m) \prod_{j=1}^m f_j(y_j)\,dy_1\cdots dy_m,
\end{align}
for all $x\not\in \cap_{j=1}^m {\rm supp}(f_j)$, where $f_1,\ldots,f_m$ are $m$ functions on $\mathbb{R}^n$ with $ \cap_{j=1}^m {\rm supp}(f_j)\not=\emptyset$.
If 
$$T: L^{p_1}(\mathbb{R}^n)\times \cdots \times L^{p_m}(\mathbb{R}^n) \to L^p(\mathbb{R}^n)$$
for some $1<p_1,\ldots,p_m$ and $p$ with $p^{-1} = \sum_{j=1}^m p_j^{-1}$, then we say $T$ is an $m$-linear Calder\'on--Zygmund operator.
According to \cite[Theorem 3]{GT}, $T$ can be extended to a bounded operator from $L^{p_1}(\mathbb{R}^n)\times \cdots \times L^{p_m}(\mathbb{R}^n) $ to $L^p(\mathbb{R}^n)$ for all for $1<p_1,\ldots,p_m$ and $p$ with $p^{-1} = \sum_{j=1}^m p_j^{-1}$.

We also define that $T$ is $mn$-homogeneous %in the $l$th component, $l=1,\ldots,m$, 
if $T$ satisfies
$$
|T(\chi_{B_0},\ldots,\chi_{B_m})(x)| \geq {C\over M^{mn}}
$$  
for $m+1$ balls $B_0=B_0(x_0,r),\ldots,B_m=B_m(x_m,r)$ satisfying
 $|y_0-y_l|\approx Mr$ for $l=1,2,\ldots,m$ and for all $x\in B_0$, where $r>0$ and $M>10$ a positive number.

Another stronger version of $mn$-homogeneous is as follows.
$$
K(x_0,\ldots,x_m) \geq {C\over M^{mn}}
$$ 
or
$$
K(x_0,\ldots,x_m) \leq -{C\over M^{mn}}
$$ 
 for $m+1$ pairwisely disjoint balls $B_0=B_0(x_0,r),\ldots,B_m=B_m(x_m,r)$ satisfying
 $|y_0-y_l|\approx Mr$ and for all $x_l\in B_l$ for $l=1,2,\ldots,m$,  where $r>0$ and $M>10$ is a positive number.  It is easy to see that this stronger version implies the version above.

In analogy with the linear case, we define the $l$th possible multilinear commutators of the $m$th multilinear Calder\'on--Zygmund operator $T$ as follows.

\begin{definition}
Suppose $T$ is an $m$-linear Calder\'on--Zygmund operator as defined above. For $l=1,2,\ldots,m$, we set
\begin{align}
[b,T]_l(f_1,\ldots,f_m)(x) := T(f_1,\ldots, bf_l,\ldots,f_m)(x) - bT(f_1,\ldots,f_m)(x).
\end{align}
\end{definition}
This is simply measuring the commutation properties in each linear coordinate separately.

Dual to the multilinear commutator, in both language and via a formal computation, we define the multilinear ``multiplication'' operators $\Pi_{l}$:

\begin{definition}\label{def of pi}
Suppose $T$ is an $m$-linear Calder\'on--Zygmund operator as defined above.  For $l=1,2,\ldots,m$,
\begin{align}
\Pi_{l}(g,h_1,\ldots,h_m)(x) := h_l T_{l}^*(h_1,\ldots, h_{l-1},g ,h_{l+1},\ldots,h_m)(x) - gT(h_1,\ldots,h_m)(x),
\end{align}
where $T_{l}^*$ is the $l$th partial adjoint of $T$, defined as
\begin{align}
 T_{l}^*(h_1,\ldots,h_m)(x):=\int_{\mathbb{R}^{mn}} K(y_l,y_1,\ldots,y_{l-1},x,y_{l+1},\ldots,y_m) \prod_{j=1}^{m}h_j(y_j) \,dy_1\cdots dy_m.
\end{align}
\end{definition}

Our main result is then the following factorization result for $H^1(\mathbb{R}^n)$ in terms of the multilinear operators $\Pi_{l}$.  Again, this is in direct analogy with the rest in the linear case obtained by Coifman, Rochberg, and Weiss in \cite{CRW}.

\begin{theorem}
\label{weakfactorization}
Suppose $1\leq l\leq m$, and $1<p_1,\ldots,p_m<\infty$ and $1\leq p<\infty$ with
$$ {1\over p_1} + \cdots + {1\over p_m} = {1\over p}.  $$
And suppose that $T$ is an $m$-linear Calder\'on--Zygmund operator, which is $mn$-homogeneous in the $l$th component. Then for every $f\in H^1(\mathbb{R}^n)$, there exists sequences $\{\lambda_s^k\}\in \ell^1$ and functions $g_s^{k}\in L^{p'}(\mathbb{R}^n)$, $h_{s,1}^{k}\in L^{p_1}(\mathbb{R}^n)$,\ldots, $h_{s,m}^{k}\in L^{p_m}(\mathbb{R}^n)$ such that
\begin{align}\label{factorization}
f=\sum_{k=1}^\infty\sum_{s=1}^{\infty} \lambda_{s}^{k}\,\Pi_{l}(g_s^{k}, h_{s,1}^{k}, \ldots, h_{s,m}^{k})
\end{align}
in the sense of $H^1(\mathbb{R}^n)$.
Moreover, we have that:
$$
\left\Vert f\right\Vert_{H^1(\mathbb{R}^n)} \approx \inf\left\{\sum_{k=1}^\infty \sum_{s=1}^{\infty} \left\vert \lambda_s^k\right\vert \left\Vert g_s^k\right\Vert_{L^{p'}(\mathbb{R}^n)}\left\Vert h_{s,1}^k\right\Vert_{L^{p_1}(\mathbb{R}^n)}\cdots \left\Vert h_{s,m}^k\right\Vert_{L^{p_m}(\mathbb{R}^n)} \right\},
$$
where the infimum above is taken over all possible representations of $f$ that satisfy \eqref{factorization}.
\end{theorem}

We then obtain the following new characterization of ${\rm BMO}(\mathbb{R}^n)$ in terms of the commutators with the multilinear Riesz transforms; again in analogy with the main results in \cite{CRW}.

\begin{theorem}
\label{c:bmo}
Let $1\leq l\leq m$.  Suppose that $T$ is an $m$-linear Calder\'on--Zygmund operator.
%\label{c:bmo}
If $b$ is in $\rm BMO(\mathbb{R}^n)$, then  the commutator
$[b,T]_l(f_1,\ldots,f_m)(x)$ is a bounded map from $L^{p_1}(\mathbb{R}^n)\times \cdots\times L^{p_m}(\mathbb{R}^n)$ to $L^{p}(\mathbb{R}^n)$ for all $1<p_1,\ldots,p_m<\infty$ and $1\leq p<\infty$, with
$$ {1\over p_1} + \cdots + {1\over p_m} = {1\over p}  $$
and with the operator norm
$$
\|[b,T]_l : L^{p_1}(\mathbb{R}^n)\times \cdots\times L^{p_m}(\mathbb{R}^n)\to L^{p}(\mathbb{R}^n)\| \leq C\|b\|_{\rm BMO(\mathbb{R}^n)}.
$$

Conversely, for $b\in \cup_{q>1} L^q_{loc}(\mathbb{R}^n)$, if $T$ is $mn$-homogeneous in the $l$th component, and  $[b,T]_l$ is bounded from $L^{p_1}(\mathbb{R}^n)\times \cdots\times L^{p_m}(\mathbb{R}^n)$ to $L^{p}(\mathbb{R}^n)$ for some $1<p_1,\ldots,p_m<\infty$ and $1\leq p<\infty$, with
$$ {1\over p_1} + \cdots + {1\over p_m} = {1\over p},  $$
then $b$ is in $\rm BMO(\mathbb{R}^n)$ and $ \|b\|_{\rm BMO(\mathbb{R}^n)} \leq C \| [b,T]_l: L^{p_1}(\mathbb{R}^n)\times \cdots\times L^{p_m}(\mathbb{R}^n)\to L^{p}(\mathbb{R}^n)\|$.
\end{theorem}

As a specific example of such operator $T$ which is an $m$-linear Calder\'on--Zygmund operator and is $mn$-homogeneous, we now recall the multilinear Riesz transforms, see \cite[Page 162]{GT} for example.  
\begin{definition}
Suppose $f_1,\ldots,f_m$ are $m$ functions on $\mathbb{R}^n$.  For $j=1,2,\ldots,m$,
\begin{align}
\vec{R}_j(f_1,\ldots,f_m)(x) := \int_{\mathbb{R}^{mn}} \vec{K}_j(x,y_1,\ldots,y_m) \prod_{s=1}^m f_s(y_s)\,dy_1\cdots dy_m,
\end{align}
where the kernel $\vec{K}_j(x,y_1,\ldots,y_m) $ is defined as 
\begin{align}
 \vec{K}_j(x,y_1,\ldots,y_m):=  { x-y_j \over |( x-y_1, \ldots, x-y_m )|^{mn+1}}.
\end{align}
\end{definition}
To be more specific, 
$$ \vec{R}_j =(R_{j}^{(1)},\ldots, R_{j}^{(n)}), $$
where for each $i=1,2,\ldots,n$, $R_{j}^{(i)}$ is the multilinear operator with the kernel
\begin{align*}
 K_j^{(i)}(x,y_1,\ldots,y_m):=  { x^{i}-y_j^{i} \over |( x-y_1, \ldots, x-y_m )|^{mn+1}}.
\end{align*}
Here $x=(x^1,\ldots,x^m)$ and $y_j =(y_j^{1} ,\ldots,y_j^{m} )$.
According to \cite[Corollary 2]{GT},  $\vec{R}_j$ is an $m$-linear Calder\'on--Zygmund operator for $j=1,2,\ldots,m$.
Moreover, we have that
\begin{align*}
|\vec{R}_j(\chi_{B_0},\ldots,\chi_{B_m})(x)| =
\left| \int_{B_1}\cdots\int_{B_m} { x-y_j \over |( x-y_1, \ldots, x-y_m )|^{mn+1}} dy_1\cdots dy_m \right| \geq {C\over M^{mn}}
\end{align*} 
for $m+1$ pairwisely disjoint balls $B_0=B_0(x_0,r),\ldots,B_m=B_m(x_m,r)$ satisfying
 $|y_0-y_l|\approx Mr$ for $l=1,2,\ldots,m$, $x\in B_0$, $r>0$, and $M>10$ a positive number.

Thus, $\vec{R}_j$ is $mn$-homogeneous.

\begin{remark}
As in Corollary 2 in \cite[Page 162]{GT}, they listed a specific multilinear Calder\'on--Zygmund kernel of the form
$$ K(x_1,\ldots,x_m) ={ \Omega\Big( { (x_1,\ldots,x_m) \over |(x_1,\ldots,x_m)| } \Big) \over |(x_1,\ldots,x_m)|^{mn}}, $$
where $\Omega$ is an integrable function with mean value zero on the sphere $\mathbb{S}^{mn-1}$ which is Lipschitz
of order $\epsilon>0$. We point out that it is possible to choose kernels of this type that satisfy the $mn$-homogeneous condition as we stated above. The Riesz transforms $\vec{R}_j$ are special examples of this form. 
\end{remark}

\begin{remark}
We remark that Theorem  \ref{c:bmo} was obtained by Chaffee in \cite{Chaffee}.  His proof uses a technique applied by Janson \cite{J}, which is different than that used here.  One advantage of the approach taken in this paper is that it provides for a constructive algorithm to produce the weak factorization of $H^1(\mathbb{R}^n)$.  As mentioned in \cite{Chaffee} it would be interesting to show the equivalence between ${\rm BMO}(\mathbb{R}^n)$ and the commutators when $p<1$.  Both the methods used there and in this paper hinge upon duality, which won't be a viable strategy when $p<1$.
\end{remark}

\section{Weak Factorization of the Hardy space $H^1(\mathbb{R}^n)$}
\setcounter{equation}{0}
\label{s:factorization}

In this section we turn to proving Theorem \ref{weakfactorization}.  We collect some facts that will be useful in proving the main result.

We first provide the following estimate of the multilinear operator $\Pi_{l}$, which is defined in Definition \ref{def of pi}.
\begin{prop}\label{t-H1 estimate of pi}
Suppose $1\leq l \leq m$. Let $1<p_1,\ldots,p_m<\infty$ and $1\leq p<\infty$ with
$$
{1\over p} ={1\over p_1} +\cdots+ {1\over p_m}.
$$
There exists a positive constant $C$ such that for any $g\in L^{p'}(\mathbb{R}^n)$ and $h_i\in L^{p_i}(\mathbb{R}^n)$, $i=1,\ldots,m$,
$$\|\Pi_{l}(g,h_1,\ldots,h_m)\|_{H^1(\mathbb{R}^n)}\leq C\|g\|_{L^{p'}(\mathbb{R}^n)}\|h_1\|_{L^{p_1}(\mathbb{R}^n)}\cdots \|h_m\|_{L^{p_m}(\mathbb{R}^n)}.$$
\end{prop}

\begin{proof}
Note that for $p_1,\ldots,p_m\in (1, \infty)$, $p\in [1,\infty)$ with ${1\over p} ={1\over p_1} +\cdots+ {1\over p_m}$, and for any $g\in L^{p'}(\mathbb{R}^n)$ and $h_i\in L^{p_i}(\mathbb{R}^n)$, $i=1,\ldots,m$,
we have
$\Pi_{l}(g,h_1,\ldots,h_m)(x) \in L^1(\mathbb{R}^n) $ by H\"older duality.  Moreover, we have
$$\int_{\mathbb{R}^n} \Pi_{l}(g,h_1,\ldots,h_m)(x)\,dx=0.$$
Hence, for $b\in {\rm BMO}(\mathbb{R}^n)$, we have
\begin{eqnarray*}
\left|\int_{\mathbb{R}^n} b(x)\Pi_{l}(g,h_1,\ldots,h_m)(x) dx \right| & = & \left|\int_{\mathbb{R}^n} g(x)[b, T]_l(h_1,\ldots,h_m)(x)dx\right|\\
& \leq & C \|h_1\|_{L^{p_1}(\mathbb{R}^n)}\cdots \|h_m\|_{L^{p_m}(\mathbb{R}^n)}\|g\|_{L^{p'}(\mathbb{R}^n)}\|b\|_{{\rm BMO}(\mathbb{R}^n)}.
\end{eqnarray*}
Here in the last equality we use \cite[Theore 3.18]{LOPTT} which provides an estimate for the multilinear commutator in terms of BMO.  Therefore, $\Pi_{l}(g,h_1,\ldots,h_m) $ is in $H^1(\mathbb{R}^n)$, with
$$\|\Pi_{l}(g,h_1,\ldots,h_m)\|_{H^1(\mathbb{R}^n)}\leq C\|g\|_{L^{p'}(\mathbb{R}^n)}\|h_1\|_{L^{p_1}(\mathbb{R}^n)}\cdots \|h_m\|_{L^{p_m}(\mathbb{R}^n)}.$$
The proof of Proposition \ref{t-H1 estimate of pi} is completed.
\end{proof}

Next, we recall a technical lemma about certain $H^1(\mathbb{R}^n)$ functions.  
\begin{lemma}\label{lemma Hardy}
Suppose $f$ is a function defined on $\mathbb{R}^n$ satisfying:  $\int_{\mathbb{R}^n} f(x)\,dx=0$, and $|f(x)|\leq \chi_{B(x_0,1)}(x)+\chi_{B(y_0,1)}(x)$, where $|x_0-y_0|:=M>10$. Then we have
\begin{align}
 \|f\|_{H^1(\mathbb{R}^n)} \leq C_n\log M.
\end{align}
\end{lemma}

We can obtain this lemma using the maximal function characterization of $H^1(\mathbb{R}^n)$, as well as the atomic decomposition characterization of $H^1(\mathbb{R}^n)$. For details of the proof, we refer to
similar versions of this lemma in \cite{DLWY} and \cite{LW}.

Suppose $1\leq l\leq m$.  Ideally, given an $H^1(\mathbb{R}^n)$-atom $a$, we would like to find functions $g\in L^{p'}(\mathbb{R}^n)$, $h_1\in L^{p_1}(\mathbb{R}^n),\ldots,$ $h_m\in L^{p_m}(\mathbb{R}^n)$ such that $\Pi_l(g,h_1,\ldots,h_m)=a$ pointwise.  While this can not be accomplished in general, the Theorem below shows that it is ``almost'' true.

\begin{theorem}
\label{thm:ApproxFactorization}
Suppose $1\leq l\leq m$. Suppose that $T$ is an $m$-linear Calder\'on--Zygmund operator, which is $mn$-homogeneous in the $l$th component. For every $H^1(\mathbb{R}^n)$-atom $a(x)$ and for all $\varepsilon>0$ and for all $1<p_1,\ldots,p_m<\infty$ and $1\leq p<\infty$, with
$$ {1\over p_1} + \cdots + {1\over p_m} = {1\over p},  $$
there exists $g\in L^{p'}(\mathbb{R}^n)$, $h_{1}\in L^{p_1}(\mathbb{R}^n)$,$\ldots$, $h_{m}\in L^{p_m}(\mathbb{R}^n)$ and a large positive number $M$ $($depending only on $\varepsilon$$)$ such that:
$$
\left\Vert a-\Pi_{l}(g,h_1,\ldots,h_m)\right\Vert_{H^1(\mathbb{R}^n)}<\varepsilon
$$
and that $\left\Vert g\right\Vert_{L^{p'}(\mathbb{R}^n)}\left\Vert h_1\right\Vert_{L^{p_1}(\mathbb{R}^n)}\cdots \left\Vert h_m\right\Vert_{L^{p_m}(\mathbb{R}^n)}\leq C M^{mn}$, where $C$ is an absolute positive constant. 
\end{theorem}

\begin{proof}
Let $a(x)$ be an $H^1(\mathbb{R}^n)$-atom, supported in $B(x_0,r)$, satisfying that
$$ \int_{\mathbb{R}^n}a(x)dx=0 \quad {\rm and} \quad \|a\|_{L^\infty(\mathbb{R}^n)}\leq r^{-n}.$$

Fix $1\leq l\leq m$ and fix $\varepsilon>0$.  Choose $M$  sufficiently large so that $$ \frac{\log M}{M^{\epsilon}} <\varepsilon,$$ where the constant $\epsilon$ appeared in the power of $M$ is from the regularity condition \eqref{regularity} of the multilinear Calder\'on--Zygmund kernel $K$. Now select $y_l\in\mathbb{R}^n$ so that $ y_{l,i}-x_{0,i}=\frac{Mr}{\sqrt{n}}$, where $x_{0,i}$ (reps. $y_{l,i}$) is the $i$th coordinate of $x_{0}$ (reps. $y_{l}$) for $i=1,2,\ldots,n$. Note that for this $y_l$, we have $|x_0-y_l|=Mr$.  Similar to the relation of $x_0$ and $y_l$, we choose $y_1$ such that $y_0$ and $y_1$ satisfies the same relationship as $x_0$ and $y_l$ do. Then by induction we choose $y_2, \ldots, y_{l-1}$, $y_{l+1},\ldots,y_m$.

We then set
\begin{eqnarray*}
 g(x) & := &\chi_{B(y_l,r)}(x),\\
 h_{j}(x) & := & \chi_{B(y_j,r)}(x),\quad j\neq l,\\
 h_{l}(x) & := & \frac{a(x)}{T_{l}^*(h_1,\ldots,h_{l-1},g,h_{l+1},\ldots,h_m)(x_0)}.\\
\end{eqnarray*}

Since $T$ is $mn$ homogeneous, and so is $T_l^*$, for the specific choice of the functions $h_1,\ldots,h_{l-1},g,h_{l+1},\ldots,h_m$ as above,  we have that there exists a positive constant $C$ such that
\begin{align}\label{claim degenerate}
\left|T_l^*(h_1,\ldots,h_{l-1},g,h_{l+1},\ldots,h_m)(x_0) \right|\geq C M^{-mn} \qquad {\rm\ for\ }\ 1\leq l\leq m.
\end{align}
%To see this, note that,
%\begin{align*}
%&T_{l}^*(h_1,\ldots,h_{l-1},g,h_{l+1},\ldots,h_m)(x_0)\\
% &= \int_{\mathbb{R}^{mn}} K(z_l,z_1,\ldots,z_{l-1},x_0,z_{l+1},\ldots,z_m) \prod_{j\neq l} h_{j}(z_j) g(z_l)\,dz_1\cdots dz_m.
%%&=\int_{\mathbb{R}^{mn}} { y_l-y_j \over |( (y_l-y_1)+ \cdots + (y_l-x_0)\cdots +(y_l-y_m) )|^{mn+1}}
%%\prod_{j\neq l} h_{j}(y_j) g(y_l)\,dy_1\cdots dy_m\\
%\end{align*}
%Since $|z_l-x_0| > |y_l-x_0| -|z_i-y_l| > Mr -r$ and $T$ is $mn$-homogeneous in the $l$th component,
%we get that
%\begin{align*}
%|T_{l}^*(h_1,\ldots,h_{l-1},g,h_{l+1},\ldots,h_m)(x_0)|\geq C {r^{nm}\over (Mr)^{mn}}= C{1\over M^{mn}}
%\end{align*}
%which implies that the claim \eqref{claim degenerate} holds.

From the definitions of the functions $g$ and $h_j$, we obtain that $\operatorname{supp}\,g=B(y_0,r)$ and $\operatorname{supp}\,h_i=B(x_0,r)$. Moreover,
$$\|g\|_{L^{p'}(\mathbb{R}^n)} \approx r^{\frac{n}{p'}}\quad\textnormal{ and }\quad \|h_i\|_{L^{p_i}(\mathbb{R}^n)} \approx  r^{\frac{n}{p_i}} $$
for $i=1,\ldots,l-1,l+1,\ldots,m$. Also we have
$$\|h_l\|_{L^{p_l}(\mathbb{R}^n)} =  \frac{1}{|T_{l}^*(h_1,\ldots,h_{l-1},g,h_{l+1},\ldots,h_m)(x_0)|} \|a\|_{L^{p_l}(\mathbb{R}^n)}\leq C M^{mn} r^{-n} r^{{n\over p_l}}, $$
where the last inequality follows from \eqref{claim degenerate}.
Hence we obtain that
\begin{align*}
 \|g\|_{L^{p'}(\mathbb{R}^n)} \|h_1\|_{L^{p_1}(\mathbb{R}^n)}\cdots \|h_m\|_{L^{p_m}(\mathbb{R}^n)} &\leq CM^{mn}r^{-n} r^{n({1\over p'}+{1\over p_1}+\cdots + {1\over p_m})}\\
&\leq  CM^{mn}.
\end{align*}
Next, we have
\begin{align*}
&a(x)-\Pi_{l}(g,h_{1},\ldots,h_{m})(x)\\
&=a(x)- \Big( h_{l} T_{l}^*(h_{1},\ldots, h_{l-1},g ,h_{l+1},\ldots,h_{l,m})(x) - gT(h_{1},\ldots,h_{m})(x) \Big)\\
&= a(x) \frac{T_{l}^*(h_1,\ldots,h_{l-1},g,h_{l+1},\ldots,h_m)(x_0)-T_{l}^*(h_1,\ldots,h_{l-1},g,h_{l+1},\ldots,h_m)(x)}
{T_{l}^*(h_1,\ldots,h_{l-1},g,h_{l+1},\ldots,h_m)(x_0)} \\
&\hskip1cm+g(x)T(h_{1},\ldots,h_{m})(x)\\
&=: W_1(x)+W_2(x).
\end{align*}

By definition, it is obvious that $W_1(x)$ is supported on $B(x_0,r)$ and $W_2(x)$ is supported on $B(y_0,r)$.  We first estimate $W_1$.  For $x\in B(x_0,r)$, we have
\begin{align*}
&|W_1(x)|\\
 &= |a(x)|\frac{|T_{l}^*(h_1,\ldots,h_{l-1},g,h_{l+1},\ldots,h_m)(x_0)-T_{l}^*(h_1,\ldots,h_{l-1},g,h_{l+1},
 \ldots,h_m)(x)|}
{|T_{l}^*(h_1,\ldots,h_{l-1},g,h_{l+1},\ldots,h_m)(x_0)|}\\
&\leq C  { \|a\|_{L^\infty(\mathbb{R}^n)} \over M^{-mn}} \int_{\prod_{j=1}^m B(y_j,r)} |K(z_l,z_1,\ldots,z_{l-1},x_0,z_{l+1},\ldots,z_m) \\
&\hskip5cm- K(z_l,z_1,\ldots,z_{l-1},x,z_{l+1},\ldots,z_m)  | \,dz_1\cdots dz_m \\
&\leq  C  M^{mn} r^{-n}  \int_{\prod_{j=1}^m B(y_j,r)} { |x_0-x|^\epsilon \over \big( \sum_{i=1,\ i\neq  l}^{m}| z_l-z_i| + | z_l-x_0|\big)^{mn+\epsilon}} \,dz_1\cdots dz_m\\
&\leq  C M^{mn}r^{-n} r^{mn} {r^\epsilon\over (Mr)^{mn+\epsilon}}\\
&\leq  C\frac{1}{ M^{\epsilon} r^n},
\end{align*}
where in the second inequality we use the regularity condition \eqref{regularity} of the multilinear kernel $K$.
Hence we obtain that $$ |W_1(x)|\leq  C\frac{1}{ M^{\epsilon} r^n} \chi_{B(x_0,r)}(x).$$

Next we estimate $W_2(x)$. From the definition of $g(x)$ and $h_l(x)$, we have
\begin{align*}
&|W_2(x)|\\
 &= \chi_{B(y_l,r)}(x) |T(h_{1},\ldots,h_{m})(x)|\\
 &=\chi_{B(y_l,r)}(x)  \frac{1}{|T_{l}^*(h_1,\ldots,h_{l-1},g,h_{l+1},\ldots,h_m)(x_0)|}\\
 &\hskip2cm\bigg| \int_{\prod_{j=1,j\neq l}^m B(y_j,r)\times B(x_0,r)} \big(K(y_1,\ldots,y_{l-1},x_0,y_{l+1},\ldots,y_m) \\
 &\hskip3cm -K(y_1,\ldots,y_{l-1},x,y_{l+1},
 \ldots,y_m) \big)a(y_l) \,dy_1\cdots dy_m \bigg|\\
&\leq C\chi_{B(y_l,r)}(x) M^{mn}  \int_{\prod_{j=1,j\neq l}^m B(y_j,r)\times B(x_0,r)}\|a\|_{L^\infty(\mathbb{R}^n)} { |x_0-x|^\epsilon \over \big( \sum_{s=1}^m |x_0-z_s| \big)^{mn+\epsilon}} \,dz_1\cdots dz_m\\
& \leq  C\chi_{B(y_l,r)}(x) M^{mn} r^{-n} \frac{r^\epsilon\cdot r^{mn}}{(Mr)^{mn+\epsilon} }\\
&= \frac{C}{M^{\epsilon} r^n },
\end{align*}
where in the second equality we use the cancelllation property of the atom $a(y_l)$
Hence we have $$ |W_2(x)|\leq  \frac{C}{M^{\epsilon} r^n } \chi_{B(y_l,r)}(x).$$

Combining the estimates of $W_1$ and $W_2$, we obtain that
\begin{align}\label{size}
 \Big|a(x)-\Pi_{l}(g,h_{1},\ldots,h_{m})(x)\Big)\Big|\leq  \frac{C}{M^{\epsilon} r^n}(\chi_{B(x_0,r)}(x)+\chi_{B(y_l,r)}(x)).
\end{align}
Next we point out that
\begin{align}\label{cancellation R}
\int_{\mathbb{R}^{n}} \Big[a(x)-\Pi_{l}(g,h_{1},\ldots,h_{m})(x)\Big) \Big] dx=0
\end{align}
since the atom $a(x)$ has cancellation and the second integral equals 0 just by the definitions of $\Pi_{l}$.

Then the size estimate \eqref{size} and the cancellation \eqref{cancellation R},  together with Lemma \ref{lemma Hardy}, imply that
$$ \Big\|a(x)-\Pi_{l}(g,h_{1},\ldots,h_{m})(x)\Big\|_{H^1(\mathbb{R}^n)} \leq C \frac{\log M}{M^{\epsilon}} <C\varepsilon. $$
This proves the result.
\end{proof}

With this approximation result, we can now prove the main Theorem \ref{weakfactorization}.

\begin{proof}[Proof of Theorem \ref{weakfactorization}]
By Proposition \ref{t-H1 estimate of pi}, we have that $$\|\Pi_{l}(g,h_1,\ldots,h_m)\|_{H^1(\mathbb{R}^n)}\leq C\|g\|_{L^{p'}(\mathbb{R}^n)}\|h_1\|_{L^{p_1}(\mathbb{R}^n)}\cdots \|h_m\|_{L^{p_m}(\mathbb{R}^n)}.$$
It is immediate that for any representation of $f$ as in \eqref{factorization}, i.e.,
\begin{align*} %\label{factorization}
f=\sum_{k=1}^\infty\sum_{s=1}^{\infty} \lambda_{s}^{k}\,\Pi_{l}(g_s^{k}, h_{s,1}^{k}, \ldots, h_{s,m}^{k}),
\end{align*}
We have that $\left\Vert f\right\Vert_{H^1(\mathbb{R}^n)}$ is bounded by
$$
 C\inf\left\{\sum_{k=1}^\infty\sum_{s=1}^{\infty} |\lambda_{s}^{k}| \|h_1\|_{L^{p_1}(\mathbb{R}^n)}\cdots \|h_m\|_{L^{p_m}(\mathbb{R}^n)}\|g\|_{L^{p'}(\mathbb{R}^n)}: f {\rm\ satisfies\ } \eqref{factorization}  \right\}.
$$

We turn to show that the other inequality holds and that it is possible to obtain such a decomposition for any $f\in H^1(\mathbb{R}^n)$.  Utilizing the atomic decomposition, for any $f\in H^1(\mathbb{R}^n)$ we can find a sequence $\{\lambda_{s}^{1}\}\in \ell^1$ and sequence of $H^1(\mathbb{R}^n)$-atoms $\{a_s^{1}\}$ so that $\displaystyle f=\sum_{s=1}^{\infty} \lambda_s^{1} a_{s}^{1}$ and $\displaystyle \sum_{s=1}^{\infty} \left\vert \lambda_s^{1}\right\vert \leq C \left\Vert f\right\Vert_{H^1(\mathbb{R}^n)}$.

We explicitly track the implied absolute constant $C$ appearing from the atomic decomposition since it will play a role in the convergence of the algorithm.  Fix $\varepsilon>0$ so that $\varepsilon C<1$.  We apply Theorem \ref{thm:ApproxFactorization} to each atom $a_{s}^{1}$.  So there exists $g_s^{1}\in L^{p'}(\mathbb{R}^n) $,  $h_{s,1}^{1}\in L^{p_1}(\mathbb{R}^n)$,\ldots, $h_{s,m}^{1}\in L^{p_m}(\mathbb{R}^n)$  with
\begin{equation*}
\left\Vert a_s^{1}-\Pi_{j,l}(g_s^{1}, h_{s,1}^{1},\ldots,h_{s,m}^{1})\right\Vert_{H^1(\mathbb{R}^n)}<\varepsilon,\quad \forall s
\end{equation*}
and
$\Vert g_s^{1}\Vert_{L^{p'}(\mathbb{R}^n)}\|h_1\|_{L^{p_1}(\mathbb{R}^n)}\cdots \|h_m\|_{L^{p_m}(\mathbb{R}^n)}\leq C(\varepsilon) $, where $C(\varepsilon)= CM^{nm}$ is a constant depending on $\varepsilon$ which we can track from Theorem \ref{thm:ApproxFactorization}.
Now note that we have
\begin{eqnarray*}
f & = & \sum_{s=1}^{\infty} \lambda_{s}^{1} a_s^{1}=   \sum_{s=1}^{\infty} \lambda_{s}^{1} \,\Pi_{l}(g_s^{1}, h_{s,1}^{1},\ldots,h_{s,m}^{1})+\sum_{s=1}^{\infty} \lambda_{s}^{1} \left(a_s^{1}-\Pi_{l}(g_s^{1}, h_{s,1}^{1},\ldots,h_{s,m}^{1})\right) \\
&=:&  M_1+E_1.
\end{eqnarray*}
Observe that we have
\begin{eqnarray*}
\left\Vert E_1\right\Vert_{H^1(\mathbb{R}^n)}  \leq  \sum_{s=1}^{\infty} \left\vert \lambda_s^{1}\right\vert \left\Vert a_s^{1}-\Pi_{l}(g_s^{1}, h_{s,1}^{1},\ldots,h_{s,m}^{1})\right\Vert_{H^1(\mathbb{R}^n)} \leq  \varepsilon \sum_{s=1}^{\infty} \left\vert \lambda_s^{1}\right\vert \leq \varepsilon C\left\Vert f\right\Vert_{H^1(\mathbb{R}^n)}.
\end{eqnarray*}
We now iterate the construction on the function $E_1$.  Since $E_1\in H^1(\mathbb{R}^n)$, we can apply the atomic decomposition in $H^1(\mathbb{R}^n)$ to find a sequence $\{\lambda_s^{2}\}\in \ell^1$ and a sequence of $H^1(\mathbb{R}^n)$-atoms $\{a_s^{2}\}$ so that $E_1=\sum_{s=1}^{\infty} \lambda_s^{2} a_{s}^{2}$ and
$$
\sum_{s=1}^{\infty} \left\vert \lambda_s^{2}\right\vert \leq C \left\Vert E_1\right\Vert_{H^1(\mathbb{R}^n)}\leq \varepsilon C^2 \left\Vert f\right\Vert_{H^1(\mathbb{R}^n)}.
$$
Again, we will apply Theorem \ref{thm:ApproxFactorization} to each atom $a_{s}^{2}$.  So there exists
$g_s^{2}\in L^{p'}(\mathbb{R}^n) $,  $h_{s,1}^{2}\in L^{p_1}(\mathbb{R}^n)$,\ldots, $h_{s,m}^{2}\in L^{p_m}(\mathbb{R}^n)$  with
\begin{equation*}
\left\Vert a_s^{2}-\Pi_{l}(g_s^{2}, h_{s,1}^{2},\ldots,h_{s,m}^{2})\right\Vert_{H^1(\mathbb{R}^n)}<\varepsilon,\quad \forall s.
\end{equation*}

We then have that:
\begin{eqnarray*}
E_1 & = & \sum_{s=1}^{\infty} \lambda_{s}^{2} a_s^{2}=   \sum_{s=1}^{\infty} \lambda_{s}^{2}\, \Pi_{l}(g_s^{2}, h_{s,1}^{2},\ldots,h_{s,m}^{2})+\sum_{s=1}^{\infty} \lambda_{s}^{2} \left(a_s^{2}-\Pi_{l}(g_s^{2}, h_{s,1}^{2},\ldots,h_{s,m}^{2})\right) \\
&:=&  M_2+E_2.
\end{eqnarray*}
But, as before, observe that
\begin{eqnarray*}
\left\Vert E_2\right\Vert_{H^1(\mathbb{R}^n)} & \leq & \sum_{s=1}^{\infty} \left\vert \lambda_s^{2}\right\vert \left\Vert a_s^{2}-\Pi_{l}(g_s^{2}, h_{s,1}^{2},\ldots,h_{s,m}^{2})\right\Vert_{H^1(\mathbb{R}^n)} \leq  \varepsilon \sum_{s=1}^{\infty} \left\vert \lambda_s^{2}\right\vert \\
&\leq& \left(\varepsilon C\right)^{2}\left\Vert f\right\Vert_{H^1(\mathbb{R}^n)}.
\end{eqnarray*}
And, this implies for $f$ that we have:
\begin{eqnarray*}
f & = & \sum_{s=1}^{\infty} \lambda_{s}^{1} a_s^{1} =   \sum_{s=1}^{\infty} \lambda_{s}^{1} \,\Pi_{l}(g_s^{1}, h_{s,1}^{1},\ldots,h_{s,m}^{1})+\sum_{s=1}^{\infty} \lambda_{s}^{1} \left(a_s^{1}-\Pi_{l}(g_s^{1}, h_{s,1}^{1},\ldots,h_{s,m}^{1})\right)\\
 & = & M_1+E_1=M_1+M_2+E_2 \\
 &= & \sum_{k=1}^{2} \sum_{s=1}^{\infty} \lambda_{s}^{k} \,\Pi_{l}(g_s^{k}, h_{s,1}^{k},\ldots,h_{s,m}^{k})+E_2.
\end{eqnarray*}

Repeating this construction for each $1\leq k\leq K$ produces functions $g_s^{k}\in L^{p'}(\mathbb{R}^n)$, $h_{s,1}^{k}\in L^{p_1}(\mathbb{R}^n),\ldots, h_{s,m}^{k}\in L^{p_m}(\mathbb{R}^n)$ with $\left\Vert g_s^{k}\right\Vert_{L^{p'}(\mathbb{R}^n)}\left\Vert h_{s,1}^{k}\right\Vert_{L^{p_1}(\mathbb{R}^n)}\cdots \left\Vert h_{s,m}^{k}\right\Vert_{L^{p_m}(\mathbb{R}^n)}\leq C(\varepsilon) $ for all $s$, sequences $\{\lambda_{s}^{k}\}\in \ell^1$ with $\left\Vert \{\lambda_{s}^{k}\}\right\Vert_{\ell^1}\leq \varepsilon^{k-1} C^k \left\Vert f\right\Vert_{H^1(\mathbb{R}^n)}$, and a function $E_K\in H^1(\mathbb{R}^n)$ with $\left\Vert E_K\right\Vert_{H^1(\mathbb{R}^n)}\leq \left(\varepsilon C\right)^{K}\left\Vert f\right\Vert_{H^1(\mathbb{R}^n)}$ so that
$$
f=\sum_{k=1}^{K}  \sum_{s=1}^{\infty} \lambda_{s}^{k} \,\Pi_{l}(g_s^{k}, h_{s,1}^{k},\ldots,h_{s,m}^{k})+E_K.
$$
Passing $K\to\infty$ gives the desired decomposition of
$$f=\sum_{k=1}^{\infty} \sum_{s=1}^\infty \lambda_{s}^{k} \,\Pi_{l}(g_s^{k}, h_{s,1}^{k},\ldots,h_{s,m}^{k}).
$$
We also have that:
$$
\sum_{k=1}^{\infty} \sum_{s=1}^\infty\left\vert \lambda_{s}^{k}\right\vert \leq \sum_{k=1}^{\infty} \varepsilon^{-1} (\varepsilon C)^{k} \left\Vert f\right\Vert_{H^1(\mathbb{R}^n)}= \frac{ C}{1-\varepsilon C}\left\Vert f\right\Vert_{H^1(\mathbb{R}^n)}.
$$
\end{proof}

Finally, we dispense with the proof of Theorem \ref{c:bmo}.

\begin{proof}[Proof of Theorem \ref{c:bmo}]
The upper bound in this theorem is contained in \cite[Theorem 3.18]{LOPTT}.  For the lower bound, suppose that $f\in H^1(\mathbb{R}^n)\cap L_c^\infty(\mathbb{R}^n)$, where $L_c^\infty(\mathbb{R}^n)$ is the subspace of $L^\infty(\mathbb{R}^n)$ consisting of functions with compact support in $\mathbb{R}^n$.
Then using the weak factorization in Theorem \ref{weakfactorization} we have that for every $b\in \cup_{q>1} L^q_{loc}(\mathbb{R}^n)$,
\begin{eqnarray*}
\left\langle b,f\right\rangle_{L^2(\mathbb{R}^n)} & = & \sum_{k=1}^{\infty} \sum_{s=1}^\infty \lambda_{s}^{k} \langle b,\Pi_{l}(g_s^{k}, h_{s,1}^{k},\ldots,h_{s,m}^{k})\rangle_{L^2(\mathbb{R}^n)}\\
 &=&  \sum_{k=1}^{\infty} \sum_{s=1}^\infty \lambda_{s}^{k} \langle g_s^{k} , [b,T]_l(h_{s,1}^{k},\ldots,h_{s,m}^{k})\rangle_{L^2(\mathbb{R}^n)}.
\end{eqnarray*}
Hence, we have that
\begin{eqnarray*}
\left\vert \left\langle b,f\right\rangle_{L^2(\mathbb{R}^n)}\right\vert &  \leq  & \sum_{k=1}^{\infty}\sum_{s=1}^\infty \left\vert \lambda_s^k\right\vert \left\Vert [b,T]_l(h_{s,1}^{k},\ldots,h_{s,m}^{k})\right\Vert_{L^p(\mathbb{R}^n)}\left\Vert g_s^k\right\Vert_{L^{p'}(\mathbb{R}^n)}\\
& \leq & \|[b,T]_l : L^{p_1}(\mathbb{R}^n)\times \cdots\times L^{p_m}(\mathbb{R}^n)\to L^{p}(\mathbb{R}^n)\| \\
&&\hskip1cm\times\sum_{k=1}^{\infty}\sum_{s=1}^\infty \left\vert \lambda_s^k\right\vert \left\Vert g_s^{k}\right\Vert_{L^{p'}(\mathbb{R}^n)} \prod_{j=1}^m \left\Vert h_{s,j}^{k}\right\Vert_{L^{p_j}(\mathbb{R}^n)}\\
& \leq  & C \|[b,T]_l : L^{p_1}(\mathbb{R}^n)\times \cdots\times L^{p_m}(\mathbb{R}^n)\to L^{p}(\mathbb{R}^n)\| \left\Vert f\right\Vert_{H^1(\mathbb{R}^n)}.
\end{eqnarray*}
By the duality between $ {\rm BMO}(\mathbb{R}^n)$ and $H^1(\mathbb{R}^n)$ we have that:
$$
\left\Vert b\right\Vert_{ {\rm BMO}(\mathbb{R}^n)}\approx \sup_{\left\Vert f\right\Vert_{H^1(\mathbb{R}^n)}\leq 1} \left\vert \left\langle b,f\right\rangle_{L^2(\mathbb{R}^n)}\right\vert\leq C\|[b,T]_l : L^{p_1}(\mathbb{R}^n)\times \cdots\times L^{p_m}(\mathbb{R}^n)\to L^{p}(\mathbb{R}^n)\|.
$$

\end{proof}

\bigskip

{\bf Acknowledgments:}    %The authors would like to thank ...
J. Li's research supported by ARC DP 160100153 and Macquarie University New Staff Grant.
B. D. Wick's research supported in part by National Science Foundation DMS grants \#1603246 and \# 1560955.

Ji Li, Department of Mathematics, Macquarie University, NSW, 2109, Australia.

\smallskip

{\it E-mail}: \texttt{ji.li@mq.edu.au}

\vspace{0.3cm}

%Brett D. Wick

%\smallskip

Brett D. Wick, Department of Mathematics, Washington University -- St. Louis, St. Louis, MO 63130-4899 USA

\smallskip

{\it E-mail}: \texttt{wick@math.wustl.edu}

\begin{thebibliography}{99}


\bibitem{Chaffee} L. Chaffee,
Characterizations of BMO through commutators of bilinear singular integral operators, arXiv:1410.4587.

\bibitem{CRW} R. R. Coifman, R. Rochberg and G. Weiss,
Factorization theorems for Hardy spaces in several variables, Ann. of Math., (2) {\bf103} (1976), 611--635.

\bibitem{DLWY}  X. T. Duong, J. Li, B. D. Wick and D. Yang,
 Factorization for Hardy spaces and characterization for BMO spaces via commutators in the Bessel setting, arXiv:1509.00079,
 accepted by Indiana University Mathematics Journal, 2016.


\bibitem{FS} C. Fefferman and E. M. Stein, $H^p$ spaces of several variables,
  Acta Math., {\bf 129} (1972), 137--193.
  
\bibitem{GT} L. Grafakos, R. H. Torres, Multilinear Calder\'on--Zygmund theory, Adv. in Math., {\bf165} (2002), 124--164.

\bibitem{J}  S. Janson, Mean oscillation and commutators of singular integral operators, Ark. Mat., {\bf16}
(1978), 263--270.

\bibitem{LOPTT} A. K. Lerner, S. Ombrosi, C. P{\'e}rez, R. Torres, R. Trujillo-Gonz{\'a}lez,
New maximal functions and multiple weights for the multilinear Calder\'on-Zygmund theory,
Adv. Math., {\bf 220}, (2009), 1222--1264.

\bibitem{LW} J. Li and B. D. Wick, Characterizations of
$H^1_{\Delta_N}(\mathbb{R}^n)$ and ${\rm BMO}_{\Delta_N}(\mathbb{R}^n)$ via weak factorizations and commutators,
arXiv:1505.04375.

\bibitem{SW} E. M. Stein and G. Weiss, On the theory of
harmonic functions of several variables. I. The theory of $H^{p}$-spaces, Acta Math. { \bf103} (1960) 25--62.

\bibitem{U} A. Uchiyama,
The factorization of $H^p$ on the space of homogeneous type,
Pacific J. Math., {\bf92} (1981), 453--468.

%\bibitem{YY}Da. Yang and Do. Yang, Characterizations of localized BMO$(\mathbb{R}^n)$ via commutators of localized Riesz transforms and fractional   integrals associated to Schr\"odinger operators,
   %Collect. Math., {\bf61} (2010), no.1, 65--79.

\end{thebibliography}
\end{document}